\documentclass{amsart}
\usepackage[utf8]{inputenc}
\usepackage[linktocpage=true,hidelinks]{hyperref}
\usepackage{setspace}
\usepackage{amssymb, amsmath, amsthm, graphicx,mathrsfs}
\usepackage{caption}
\usepackage[noadjust]{cite}
\usepackage{mathtools,xcolor, tikz}
\usepackage{titling}
\thanksmarkseries{arabic}

\usepackage{enumerate}

\usepackage[foot]{amsaddr}

\usepackage{concmath}
\usepackage{ccfonts}
\usepackage{eulervm}

\usepackage{soul}

\newtheorem{thm}{Theorem}
\newtheorem{cor}[thm]{Corollary}

\newtheorem{conj}[thm]{Conjecture}
\newtheorem{prop}[thm]{Proposition}

\newtheorem{ques}{Question}

\newcommand{\T}{\mathsf{T}}

\renewcommand{\l}{\left}
\renewcommand{\r}{\right}

\newcommand{\eps}{\epsilon}

\newcommand{\symdif}{\mathbin{\triangle}}

\title{On odd covers of cliques and disjoint unions}
\author{Calum Buchanan$^1$}
\address{$^1$Dept.\:of Mathematics \& Statistics, University of Vermont, Burlington, VT, USA}
\author{Alexander Clifton$^2$}
\address{$^2$Discrete Mathematics Group, Institute for Basic Science, Daejeon, South Korea}
\author{Eric Culver$^3$}
\address{$^3$Dept.\:of Mathematics, Brigham Young University, Provo, UT, USA}
\author{P\'{e}ter Frankl$^4$}
\address{$^4$Alfr\'{e}d R\'{e}nyi Institute of Mathematics, Budapest, Hungary}
\author{Jiaxi Nie$^5$}
\address{$^5$Shanghai Center of Mathematical Sciences, Fudan University, Shanghai, China}
\author{Kenta Ozeki$^6$}
\address{$^6$Faculty of Environment and Information Sciences, Yokohama National University, Yokohama, Japan}
\author{Puck Rombach$^1$}
\author{Mei Yin$^7$}
\address{$^7$Dept.\:of Mathematics, University of Denver, Denver, CO, USA}

\date{\today}

\email{calum.buchanan@uvm.edu}
\email{yoa@ibs.re.kr}
\email{eric.culver@mathematics.byu.edu}
\email{frankl.peter@renyi.hu}
\email{jiaxi\_nie@fudan.edu.cn}
\email{ozeki-kenta-xr@ynu.ac.jp}
\email{puck.rombach@uvm.edu}
\email{mei.yin@du.edu}

\begin{document}

\begin{abstract}

Babai and Frankl posed the ``odd cover problem" of finding the minimum cardinality of a collection of complete bipartite graphs such that every edge of the complete graph of order $n$ is covered an odd number of times.
In a previous paper with O'Neill, some of the authors proved that this value is always $\lceil n / 2 \rceil$ or $\lceil n / 2 \rceil + 1$ and that it is the former whenever $n$ is a multiple of $8$.
In this paper, we determine this value to be $\lceil n / 2 \rceil$ whenever $n$ is odd or equivalent to $18$ modulo $24$.
We also further the study of odd covers of graphs which are not complete, wherein edges are covered an odd number of times and nonedges an even number of times by the complete bipartite graphs in the collection.
Among various results on disjoint unions, we find the minimum cardinality of an odd cover of a union of odd cliques and of a union of cycles.
\end{abstract}

\maketitle

\section{Introduction}

A celebrated result of Graham and Pollak~\cite{GP} states that one needs at least $n-1$ bicliques (complete bipartite graphs) to cover every edge of $K_n$ (the complete graph on $n$ vertices) exactly once. Given a finite simple graph $G$, let $bp(G)$ denote the minimum integer such that $G$ can be partitioned into $bp(G)$ bicliques. Thus Graham-Pollak states that $bp(K_n)=n-1$.

Babai and Frankl posed the following question in~\cite{bf}:
{\em What is the minimum number of bicliques needed to cover every edge of $K_n$ an odd number of times?}
They called this the ``odd cover problem," generalized in~\cite{beaudrap, buchanan2022odd} as follows. Let $G$ be a finite simple graph. An {\em odd cover} of $G$ is a collection of bicliques on subsets of the vertex set $V(G)$ which cover each edge of $G$ an odd number of times and each nonedge of $G$ an even number of times. Let $b_2(G)$ denote the minimum cardinality of an odd cover of $G$. In this language, Babai and Frankl asked for the value $b_2(K_n)$. 

Clearly, we always have $b_2(G)\le bp(G)$. The smallest order of a graph for which this inequality is strict is $5$. Among these are the butterfly (two triangles sharing a common vertex) and $K_5$. Indeed, we have $b_2(K_5)=3$ (see Figure~\ref{fig:K5}) while $bp(K_5)=4$ by Graham-Pollak.

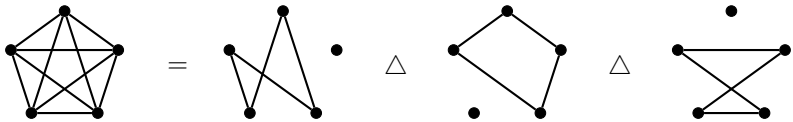
\begin{figure}[h]
\begin{center}
\begin{tikzpicture}
[every node/.style={circle, draw=black!100, fill=black!100, inner sep=0pt, minimum size=4pt},
every edge/.style={draw, black!100, thick},
scale=.75]

\foreach \i in {1,...,5}
{
\node (\i) at (\i*360/5+18:1) {};
}
\draw[black!100,thick] (1) -- (2) -- (3) -- (4)-- (5)-- (1);
\draw[black!100,thick] (1) -- (3);
\draw[black!100,thick] (1) -- (4);
\draw[black!100,thick] (2) -- (4);
\draw[black!100,thick] (2) -- (5);
\draw[black!100,thick] (3) -- (5);

\node[draw=none,fill=none] at (2,0) {$=$};
\end{tikzpicture}
\quad
\begin{tikzpicture}
[every node/.style={circle, draw=black!100, fill=black!100, inner sep=0pt, minimum size=4pt},
every edge/.style={draw, black!100, thick},
scale=.75]

\foreach \i in {1,...,5}
{
\node (\i) at (\i*360/5+18:1) {};
}
\draw[black!100,thick] (1) -- (3);
\draw[black!100,thick] (1) -- (4);
\draw[black!100,thick] (2) -- (3);
\draw[black!100,thick] (2) -- (4);

\node[draw=none,fill=none] at (2,0) {$\triangle$};
\end{tikzpicture}
\quad
\begin{tikzpicture}
[every node/.style={circle, draw=black!100, fill=black!100, inner sep=0pt, minimum size=4pt},
every edge/.style={draw, black!100, thick},
scale=.75]

\foreach \i in {1,...,5}
{
\node (\i) at (\i*360/5+18:1) {};
}
\draw[black!100,thick] (1) -- (2) -- (4) -- (5) -- (1);

\node[draw=none,fill=none] at (2,0) {$\triangle$};
\end{tikzpicture}
\quad
\begin{tikzpicture}
[every node/.style={circle, draw=black!100, fill=black!100, inner sep=0pt, minimum size=4pt},
every edge/.style={draw, black!100, thick},
scale=.75]

\foreach \i in {1,...,5}
{
\node (\i) at (\i*360/5+18:1) {};
}
\draw[black!100,thick] (2) -- (4) -- (3) -- (5) -- (2);
\end{tikzpicture}
\end{center}

\caption{An odd cover of $K_5$. }
\label{fig:K5}
\end{figure}

If $H$ is a graph obtained from $G$ by deleting a vertex, then we have $bp(G)\le bp(H)+1$ by taking the optimal biclique partition of $H$ plus a star centered at the deleted vertex. Thus, for any $n$-vertex graph $G$, we have $b_2(G)\le bp(G)\le n-\alpha(G)$, where $\alpha(G)$ is the independence number of $G$ (the maximum 
number of pairwise nonadjacent vertices in $G$). It would be interesting to determine whether there exists $\epsilon>0$ such that, for $n$ sufficiently large, $b_2(G)<(1-\epsilon)n$ for every $n$-vertex graph $G$. In fact, we do not know of any counterexamples to the inequality $b_2(G)\le n/2 +1$. 

Let $A_G$ denote the adjacency matrix of $G$ and $r_2(G)$ its rank over $\mathbb{F}_2$.
All ranks in this paper are taken over $\mathbb{F}_2$, and we simply call this the {\em rank of $G$} (or of $A_G$).
A simple lower bound 
\begin{equation}\label{eq:ranklower}
    b_2(G)\ge \frac{r_2(G)}{2}
\end{equation}
follows from the fact that every complete bipartite graph has rank $2$. For later purposes, let us provide a different proof. Take $M$ to be the $n \times 2k$ incidence matrix for a minimum odd cover of an $n$-vertex graph $G$, whose columns correspond to the partite sets $X_1, Y_1, X_2, Y_2, \ldots, X_k, Y_k$ in the odd cover and whose rows correspond to the vertices $v_1, \ldots, v_n$ of $G$.
One can check that
\begin{equation}\label{eq:matrixdecomp}
    A_G = M A_k M^\T
\end{equation}
where $A_k$ denotes the adjacency matrix of a perfect matching on $2k$ vertices; that is, $A_k$ is the direct sum of $k$ copies of $\left( \begin{smallmatrix} 0 & 1 \\ 1 & 0 \end{smallmatrix} \right)$.
It follows that the rank of $A_G$ is at most the rank of $M$, which is at most $2 b_2(G)$.
Inspired by this fact, we say that an odd cover of $G$ of cardinality $r_2(G)/2$ is a \emph{perfect odd cover} of $G$. 

One can always find an $n \times r_2(G)$ matrix $M$ which allows for a decomposition of $A_G$ as in~\eqref{eq:matrixdecomp}~\cite{godsil2001chromatic}.
Not every matrix $M$ satisfying~\eqref{eq:matrixdecomp}, however, is an incidence matrix for a perfect odd cover of $G$.
A correspondence is established in~\cite{buchanan2022subgraph, buchanan2022odd} between such a matrix $M$ and a collection of $r_2(G) / 2$ tricliques (complete tripartite graphs) and bicliques which cover each edge of $G$ an odd number of times and each nonedge an even number of times.
We call such a collection a {\em $\mathcal{T}$-odd cover} of $G$. These will be useful in Section~\ref{sec:general} when we prove certain properties possessed by perfect odd covers.


Babai and Frankl's original odd cover problem remains only partially solved.
In \cite{buchanan2022odd}, it was shown that perfect odd covers do not exist for complete graphs of odd order, but do exist for those of order divisible by $8$.

\begin{thm}[\cite{buchanan2022odd}]\label{thm:cliques from our first paper}
    For all $n\ge 3$,
    $$
\l\lceil\frac{n}{2}\r\rceil\le b_2(K_n)\le \l\lceil\frac{n}{2}\r\rceil+1.
$$
In particular, for any positive integer $k$, $K_{8k}$ has a perfect odd cover, and hence,
\[
\begin{aligned}
&b_2(K_{8k-1})=b_2(K_{8k})=4k,\\
&b_2(K_{8k+1})=4k+1.
\end{aligned}
\]
\end{thm}

Further, the authors of~\cite{buchanan2022odd} made the following conjecture,
\begin{conj}\label{conjecture:main}
If $k$ is a positive integer, then
    \begin{enumerate}[{\it (i)}]
        \item $b_2(K_{2k+1})=k+1$, and
        \item $b_2(K_{2k})=k+1$ whenever $k\equiv 2 \pmod{4}$ or $k \equiv 3 \pmod{4}$.
    \end{enumerate}
        
\end{conj}



In Theorem \ref{thm:oddclique}\footnote{While writing a first draft of this work, we became aware of parallel, independent work by Leader and Tan. Aside from both addressing the odd cover problem, the results of the two papers are essentially disjoint, except that Theorem \ref{thm:oddclique} and Proposition \ref{prop:tomon3^k} appeared in Section 2 of \cite{LT24} with different proofs.}, we resolve the odd cover problem for cliques of odd order, demonstrating an upper bound by generalizing the construction in Figure~\ref{fig:K5} to answer Conjecture \ref{conjecture:main}{\em (i)} in the affirmative. 
Note that Theorem \ref{thm:cliques from our first paper} reduces Babai and Frankl's odd cover problem for cliques of even order to determining which ones have perfect odd covers. In Theorem \ref{thm:18mod24}, we show that $K_n$ has a perfect odd cover whenever $n\equiv 18\pmod{24}$.  We also find alternative perfect odd covers for infinitely many even cliques via a finite field construction, and we demonstrate properties that any perfect odd cover of an even clique must have (Theorem \ref{theorem:even_partial}). 
In the process of trying to construct the smallest possible odd cover for $K_n$ with $n$ odd, we naturally encounter the question of finding an odd cover for a disjoint union of two or more graphs. This inspires our investigation of $b_2$ for a disjoint union of cliques (Theorem \ref{thm:oddcliqueunion}), a disjoint union of cycles (Theorem \ref{prof:cycleunion}), and the general relationship between $b_2$ for a disjoint union of graphs versus for the constituent graphs (Proposition \ref{prop:jamthingsin}). 

Let us introduce some notations.
We let $V(G)$ and $E(G)$ respectively denote the vertex set and the edge set of a graph $G$. We let $G[S]$ denote the induced subgraph of $G$ on a subset $S$ of $V(G)$. For a second graph $H$, we let $G\symdif H$ denote the graph on $V(G) \cup V(H)$ whose edge set is the symmetric difference of $E(G)$ and $E(H)$. We denote by $G+H$ the disjoint union $G$ and $H$ and by $tG$ the disjoint union of $t$ copies of $G$. 
We denote by $N(v)$ the {\em open neighborhood} of a vertex $v$ of $G$, that is, $N(v) = \{u \mid uv \in E(G)\}$, and by $N[v]$ the {\em closed neighborhood} of $v$, that is, $N[v] = N(v) \cup \{v\}$. 


The rest of the paper is organized as follows. In Section \ref{sec:general}, we establish several general results on odd covers which will be applied in the later sections. Of chief importance is Corollary \ref{cor:evenint/evencore} which provides a necessary condition on perfect odd covers which we use both for disjoint unions of cycles and for even cliques. In Section \ref{sec:disjoint}, we show that $b_2(2G)\le |V(G)|$ (Theorem~\ref{thm:disjoint_copies}), a result that will prove crucial in establishing a tight upper bound for $b_2(K_{2n+1})$. We also show that $b_2(G)= r_2(G) / 2+1$ for any disjoint union of cycles containing an odd cycle and find surprising examples of nontrivial graphs $G$ and $H$ satisfying $b_2(G+H)=b_2(H)$. In Section \ref{sec:clique}, we consider odd covers of complete graphs. In Subsection \ref{sec:oddclique}, we completely determine $b_2(K_{2n+1})$, resolving part {\em (i)} of Conjecture~\ref{conjecture:main}, and extend the result to disjoint unions of odd cliques. In Subsection \ref{sec:evenclique}, we demonstrate a new positive density family of even cliques for which there is a perfect odd cover and provide conditions that any perfect odd cover of $K_{2n}$ must satisfy.

\section{Bases and Even Cores}\label{sec:general}




We begin this section with a theorem concerning the matrix decomposition~\eqref{eq:matrixdecomp}.

\begin{thm}\label{thm:matrixdecomp}
    Let $G$ be a graph of order $n$ and rank  $2k$, and let $M \in \mathbb{F}_2^{n \times 2k}$ be such that $M A_k M^\T = A_G$, where $A_k$ is the direct sum of $k$ copies of $\left( \begin{smallmatrix} 0 & 1 \\ 1 & 0 \end{smallmatrix} \right)$.
    For any subset $S$ of $\{1, \ldots, n\}$, the set of rows in $M$ indexed by $S$ is independent if and only if the set of rows in $A_G$ indexed by $S$ is independent.
\end{thm}

\begin{proof}
    Let $G$ and $M$ be as described.
    First, suppose that a subset $\{ m_s : s \in S \}$ of rows of $M$ is such that $\sum_{s \in S} m_s = \Vec{0}$.
    If $\Vec{s}$ denotes the $1 \times n$ incidence vector for $S$, then $\Vec{s} M = \Vec{0}$, and thus $\Vec{s} A_G = \Vec{0}$ by~\eqref{eq:matrixdecomp}.
    It follows that $\{ m_s : s \in S \}$ is independent whenever the set $\{ a_s : s \in S \}$ of rows in $A_G$ is independent.

    On the other hand, suppose that $\{ m_s : s \in S \}$ is independent.
    This set is contained in a basis for the row space of $M$ which induces a $2k \times 2k$ full-rank submatrix $M_B$.
    After a reordering of the vertices of $G$ and the rows of $M$, we can write
    \[
    M =
    \begin{pmatrix}
        M_B \\ R M_B
    \end{pmatrix}
    =
    \begin{pmatrix}
        I \\ R
    \end{pmatrix}
    M_B.
    \]
    Letting $B = M_B A_k M_B^\T$, we have
    \[
    A_G = 
    \begin{pmatrix}
		I \\ R 
    \end{pmatrix} 
    B
    \begin{pmatrix}
		I & R^T
    \end{pmatrix}
    =
    \begin{pmatrix}
        B & B R^\T \\
        R B & R B R^\T
    \end{pmatrix}
    \]
    by~\eqref{eq:matrixdecomp}.
    It follows that the first $2k$ rows in $A_G$, those corresponding to $B$, span the row space of $A_G$, and these contain every $a_s$ for $s \in S$. 
    Since $r_2(G) = 2k$, the proof is complete.
\end{proof}

We now discuss the relevance of Theorem~\ref{thm:matrixdecomp} to perfect odd covers as well as several of its corollaries.
A set of rows in $A_G$ which sum to $\Vec{0}$ over $\mathbb{F}_2$ has a graph-theoretic interpretation: 
every vertex in $G$ has an even number of neighbors in the corresponding subset of vertices.
If such a subset of vertices is nonempty, we call it an {\em even core} in $G$.
Just as a set of vectors which sum to $\Vec{0}$ can be partitioned into minimally dependent sets of vectors, every even core can be partitioned into minimal even cores.

Recall that an $n \times r_2(G)$ matrix $M$ satisfying~\eqref{eq:matrixdecomp} corresponds to a minimum $\mathcal{T}$-odd cover of $G$.

\begin{cor}\label{cor:dependencies}
    Let $M$ be the incidence matrix for a minimum $\mathcal{T}$-odd cover of a graph $G$.
    A nonempty subset of $V(G)$ is an even core in $G$ if and only if the corresponding rows in $M$ sum to the zero vector.
\end{cor}

Now, if $G$ is a graph with a perfect odd cover $\mathcal{O}$, then its incidence matrix $M$ is an $n \times r_2(G)$ matrix (whose columns are paired into $n\times 2$ matrices so that no row is $(1,1)$) satisfying~\eqref{eq:matrixdecomp}.
In this case, a set of rows in $M$ which sums to $\Vec{0}$ corresponds to a set of vertices in $G$ which has even intersection with every partite set in $\mathcal{O}$.
This gives us the following corollary, which will be useful for our study of disjoint unions of cycles in Section~\ref{sec:disjoint} and to prove a necessary condition for an even clique to have a perfect odd cover in Section \ref{sec:clique}.

\begin{cor}\label{cor:evenint/evencore}
    If $\mathcal{O}$ is a perfect odd cover of a graph $G$, then the even cores in $G$ are precisely those nonempty subsets of $V(G)$ which have even intersection with both partite sets of every biclique in $\mathcal{O}$.
    Hence, if any subgraph of $G$ induced by an even core has an odd number of edges, then $b_2(G) > r_2(G) / 2$.
\end{cor}

\begin{proof}
    If $W$ is an even core such that $G[W]$ has an odd number of edges, then any odd cover of $G$ must contain some biclique $(X,Y)$ where both $|X\cap W|,|Y\cap W|$ are odd. Thus, such an odd cover cannot be a perfect odd cover, so $b_2(G)\ge \frac{r_2(G)}{2}+1$.
\end{proof}


As a final observation on this relationship between even cores and perfect odd covers, suppose that $G$ is a graph with a perfect odd cover $\mathcal{O}$ and an even core $W$.
For any $w \in W$ and any $(X,Y) \in \mathcal{O}$, since each of $|X \cap W|$ and $|Y \cap W|$ is even and $X \cap Y = \varnothing$, at least one of $|X \cap (W - w)|$ or $|Y \cap (W-w)|$ is even.
In the case that $W$ is a minimal even core of cardinality $3$, this implies that no biclique in $\mathcal{O}$ builds an edge between any pair of vertices in $W$.
In particular, $W$ forms an independent set in $G$.

We conclude this section by noting that, roughly, the subsets of vertices which index bases for the row space of $A_G$ are complete systems of distinct representatives for the partite sets in any perfect odd cover of $G$, should one exist.
We recall that a {\em complete system of distinct representatives}, or transversal, of a set system $\mathcal{A} \subseteq 2^{V}$ is a set of $|\mathcal{A}|$ elements of $V$, each one contained in a distinct set in $\mathcal{A}$.
Note that, when a graph $G$ has a perfect odd cover, the columns of its incidence matrix $M$ are linearly independent (see equation~\eqref{eq:matrixdecomp}).
If $G$ is full-rank, then $M$ is a square matrix.
One can show in this case that $V(G)$ constitutes a complete system of distinct representatives for the partite sets in the perfect odd cover.
We prove a more general theorem.


\begin{thm}\label{thm:basesareSDRs}
    Let $G$ be a graph with $r_2(G)=2k$ and a perfect odd cover $\{(X_i, Y_i), \ldots, (X_k, Y_k)\}$.
    If a set of rows in $A_G$ is a basis for its row space, then the corresponding set of vertices in $G$ is a complete system of distinct representatives for the collection of partite sets $\{X_1, Y_1, \ldots, X_k, Y_k\}$.
\end{thm}
\begin{proof}
Suppose that a set of rows in $A_G$, corresponding to a subset $B$ of $V(G)$, is a basis for the row space of $A_G$.
For $1\le m\le 2k$, any set of $m$ vertices in $B$ collectively appears in at least $m$ partite sets of the perfect odd cover. This is due to the fact that if a vertex $v$ is contained only within a subset of a fixed collection of $m-1$ partite sets, then the corresponding row of $A_G$ is spanned by the indicator vectors of the $m-1$ opposite partite sets, so these $m$ independent row vectors would lie in a dimension $m-1$ subspace of $\mathbb{F}_2^n$, giving a contradiction. Thus, by Hall's Marriage Theorem \cite{Hal35}, there is an ordering $x_1, y_1, \ldots, x_k, y_k$ of the vertices in $B$ so that $x_i \in X_i$ and $y_i \in Y_i$ for all $i \in \{1, \ldots, k\}$.
\end{proof}

\section{Disjoint Unions}\label{sec:disjoint}
In this section, we investigate the parameter $b_2$ on disconnected graphs.
Theorem \ref{thm:disjoint_copies} provides an upper bound on $b_2(2G)$, which is tight when $G$ has full rank (for example when $G$ is the disjoint union of even cliques), and it will later prove useful for computing $b_2$ for odd cliques. We also determine $b_2$ exactly for a disjoint union of cycles and provide several examples where $b_2(G+H)=b_2(H)$ for a nontrivial $G$.

\begin{thm}\label{thm:disjoint_copies}
Let $G$ be the disjoint union of two isomorphic graphs $H_1$ and $H_2$. Let $V=\{v_1,\dots,v_k\}$ and $W=\{w_1,\dots,w_k\}$ be the vertex sets of $H_1$ and $H_2$ respectively such that the function $f$ with $f(v_i)=w_i$, $1\le i\le k$, is a graph isomorphism between $H_1$ and $H_2$. Then there exist $k$ complete bipartite graphs $B_1,~\dots,~B_k$ with parts $(X_1,Y_1),~\dots,~(X_k,Y_k)$ respectively such that:
\begin{enumerate}[{\it (i)}]
    \item $B_1,~\dots,~B_k$ form an odd cover of $G$;
    \item for all $1\le i\le k$, $X_i=\{v_i,~w_i\}$. 
\end{enumerate}
\end{thm}

\begin{proof}
    We proceed by induction on $k$. When $k=1$, this is trivial ($Y_1=\emptyset$). When $k\ge 2$, we assume that the theorem is true for $k-1$. Let $G'$ be the induced subgraph of $G$ on $\{v_1,\dots,v_{k-1}\}\cup\{w_1,\dots,w_{k-1}\}$. Then by inductive assumption, there exist bipartite graphs $B'_1,~\dots,~B'_{k-1}$ with parts $(X'_1,~Y'_1),~\dots,~(X'_{k-1},~Y'_{k-1})$ respectively such that: (i) $B'_1,~\dots,~B'_{k-1}$ form an odd cover of $G'$; and (ii) for all $1\le i\le k-1$, $X'_i=\{v_i,~w_i\}$. 
    
    Now we construct an odd cover of $G$, $(B_1,~\dots,~B_k)$, from $(B'_1,~\dots,~B'_{k-1})$ as follows. First, we let $X_i=\{v_i,~w_i\}$ for all $1\le i\le k$ as required. Then, for any $1\le j\le k-1$, if $v_jv_k\in E(G)$,  we put $v_k$ into $Y_j$. More precisely, let
    $$
    \begin{aligned}
        &Y_j=Y'_j\cup\{v_k\}~&\text{if~$v_jv_k\in E(G);$}\\
        &Y_j=Y'_j~&\text{if~$v_jv_k\not\in E(G).$}
    \end{aligned}
    $$
    Lastly, for any $1\le j\le k-1$, if $v_jv_k\in E(G)$ (hence $w_jw_k\in E(G)$), then we put $w_j$ into $Y_k$. More precisely, let
    $$
    Y_k=\{w_j:~w_jw_k\in E(G),~1\le j\le k-1\}.
    $$
    It is not hard to check that now $B_1,~\dots,~B_k$ form an odd cover of $G$. This completes the proof.
\end{proof}


In \cite{buchanan2022odd}, it is shown that $b_2(C_{2n+1})=n+1$. That is, $b_2$ of an odd cycle is $1$ more than the rank bound. In fact, Corollary 5 provides a short proof of the lower bound. We show that this can be generalized for any union of cycles containing at least one odd cycle.

\begin{figure}
\centering
\begin{tikzpicture}
[every node/.style={circle, draw=black!100, fill=black!100, inner sep=0pt, minimum size=4pt},
every edge/.style={draw, black!100, thick},
scale=.75]

\foreach \i in {1,...,7}
{
\node (\i) at (\i*360/7:1) {};
}
\draw[black!100,thick] (1) -- (2) -- (3) -- (4) -- (5) -- (6) -- (7) -- (1);

\node[draw=none,fill=none] at (2,0) {$=$};
\end{tikzpicture}
\quad
\begin{tikzpicture}
[every node/.style={circle, draw=black!100, fill=black!100, inner sep=0pt, minimum size=4pt},
every edge/.style={draw, black!100, thick},
scale=.75]

\foreach \i in {1,...,7}
{
\node (\i) at (\i*360/7:1) {};
}
\draw[black!100,thick] (1) -- (2) -- (6) -- (7) -- (1);

\node[draw=none,fill=none] at (2,0) {$\triangle$};
\end{tikzpicture}
\quad
\begin{tikzpicture}
[every node/.style={circle, draw=black!100, fill=black!100, inner sep=0pt, minimum size=4pt},
every edge/.style={draw, black!100, thick},
scale=.75]

\foreach \i in {1,...,7}
{
\node (\i) at (\i*360/7:1) {};
}
\draw[black!100,thick] (2) -- (3) -- (5) -- (6) -- (2);

\node[draw=none,fill=none] at (2,0) {$\triangle$};
\end{tikzpicture}
\quad
\begin{tikzpicture}
[every node/.style={circle, draw=black!100, fill=black!100, inner sep=0pt, minimum size=4pt},
every edge/.style={draw, black!100, thick},
scale=.75]

\foreach \i in {1,...,7}
{
\node (\i) at (\i*360/7:1) {};
}
\draw[black!100,thick] (3) -- (4) -- (5) -- (3);
\end{tikzpicture}
\caption{Extending an odd cover of $C_3$ to an odd cover of $C_7$}
\label{fig:C7}
\end{figure}
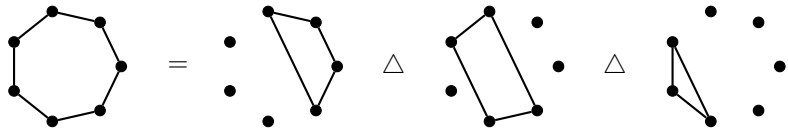

\begin{thm}\label{prof:cycleunion}
    When $t\ge 1, \ell\ge 0$, \[b_2(C_{2n_1+1}+\dots+C_{2n_t+1}+C_{2m_1}+\dots+C_{2m_\ell})=(n_1+\dots+n_t)+(m_1+\dots+m_\ell)-\ell+1.\]
\end{thm}
\begin{proof}
Let $G:=C_{2n_1+1}+\dots+C_{2n_t+1}+C_{2m_1}+\dots+C_{2m_\ell}$. 
Note that $r_2(G)=\sum_{i=1}^t 2n_i+\sum_{i=1}^{\ell} (2m_i-2)=2 \sum_{i=1}^t n_i+2\sum_{i=1}^{\ell}m_i-2\ell$.
The component $C_{2n_1+1}$ is an even core of $G$ with an odd number of edges, so by Corollary \ref{cor:evenint/evencore}, we obtain $b_2(G)\ge \sum_{i=1}^t n_i+\sum_{i=1}^{\ell}m_i-\ell+1$.
Now we establish the upper bound.
The subgraph $C_{2m_1}+\dots+C_{2m_\ell}$ is bipartite, so from \cite[Theorem~4.4]{buchanan2022odd}, we know that $b_2(C_{2m_1}+\dots+C_{2m_\ell})=r_2(C_{2m_1}+\dots+C_{2m_\ell})/2=(m_1+\dots+m_\ell)-\ell$. Then, from \cite[Proposition~3.7]{buchanan2022odd}, we know that there exists an odd cover of $tC_3$ consisting of $t+1$ bicliques. An odd cover of $C_3$ can be turned into an odd cover of $C_{2n+1}$ for $n\ge 2$ by sequentially taking the symmetric difference with $n-1$ $K_{2,2}$'s; see Figure \ref{fig:C7}. Therefore $b_2(C_{2n_1+1}+C_{2n_2+1}+\dots+C_{2n_t+1})\le (t+1)+\sum_{i=1}^t (n_j-1)=1+\sum_{i=1}^t n_j$. Combining these constructions yields the desired upper bound.
\end{proof}


\subsection{Oddities}\label{sub:oddities}

Having examined some specific cases, we now inquire into the general behavior of the parameter $b_2$ under disjoint unions. In particular, for a given graph $G$, we would like to know if one can always find a graph $H$ so that $b_2(G + H) = b_2(H)$. In this section, we answer this question in the affirmative for $K_2$ and $K_3$. An example graph $H$ when $G = K_2$ is shown in Figure~\ref{fig:plusk2}.


\begin{prop} \label{prop:jamthingsin}There exist graphs $H_1$ and $H_2$ satisfying:
    \begin{enumerate}[(i)]
        \item\label{prop:item:K2}  $b_2(H_1)=4=b_2(H_1+K_2)$. 
        \item\label{prop:item:K3}  $b_2(H_2)=5=b_2(H_2+K_3)$. 
    \end{enumerate}
\end{prop}

\begin{proof}
Recall from \cite{buchanan2022odd} the graph $B_k$ with vertex set consisting of all length $k$ strings in $\{0,1,\epsilon\}^k$, where vertices $u$ and $v$ are adjacent if and only if the number of places where one contains $0$ and the other contains $1$ is odd. A twin-free graph $G$ satisfies $b_2(G)\le k$ if and only if $G$ is an induced subgraph of $B_k$.

First, we prove item~{\em (\ref{prop:item:K2})}.
We can embed $K_2$ into $B_4$ as the vertices $000\epsilon$ and $1000$. We then choose $H_1$ to be the induced subgraph of $B_4$ on the vertex set of all vertices nonadjacent to  $000\epsilon$ and $1000$. By construction, $b_2(H_1+K_2)\le 4$, and indeed we can check computationally that $b_2(H_1)\ge 4$. Since $b_2(H_1+K_2)\ge b_2(H_1)$, we have the desired result.
It is also possible to choose $H_1$ as a smaller induced subgraph of $B_4$. One option, shown in Figure \ref{fig:plusk2}, comes from using the vertex set
\[\{\epsilon00\epsilon,\eps11\eps,1101,1011,\eps0\eps\eps,11\eps1,1\eps11,\eps\eps0\eps,0111,\eps110\}.\]
  
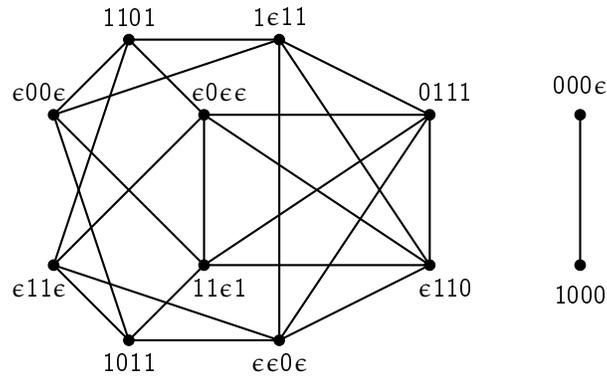
\begin{figure}
    \centering
    \begin{tikzpicture}
    [vertex/.style={circle, draw=black!100, fill=black!100, inner sep=0pt, minimum size=4pt},
    every edge/.style={draw, black!100, thick}]
    
    \draw [black!100,thick] (7,1) -- (7,-1);
    
    \draw [black!100,thick] (5,1) -- (5,-1);
    \draw [black!100,thick] (3,2) -- (3,-2);
    \draw [black!100,thick] (2,1) -- (2,-1);
    
    \draw [black!100,thick] (5,1) -- (3,2);
    \draw [black!100,thick] (5,1) -- (3,-2);
    \draw [black!100,thick] (5,1) -- (2,1);
    \draw [black!100,thick] (5,1) -- (2,-1);
    
    \draw [black!100,thick] (5,-1) -- (3,2);
    \draw [black!100,thick] (5,-1) -- (3,-2);
    \draw [black!100,thick] (5,-1) -- (2,1);
    \draw [black!100,thick] (5,-1) -- (2,-1);
    
    \draw [black!100,thick] (3,2) -- (1,2);
    \draw [black!100,thick] (3,2) -- (0,1);
    \draw [black!100,thick] (3,-2) -- (1,-2);
    \draw [black!100,thick] (3,-2) -- (0,-1);
    
    \draw [black!100,thick] (2,1) -- (1,2);
    \draw [black!100,thick] (2,1) -- (0,-1);
    \draw [black!100,thick] (2,-1) -- (1,-2);
    \draw [black!100,thick] (2,-1) -- (0,1);
    
    \draw [black!100,thick] (0,1) -- (1,2);
    \draw [black!100,thick] (0,1) -- (1,-2);
    \draw [black!100,thick] (0,-1) -- (1,2);
    \draw [black!100,thick] (0,-1) -- (1,-2);
    
    \node (e000) [vertex] at (7,1) {};
    \node (1000) [vertex] at (7,-1) {};
    
    \node (0111) [vertex] at (5,1) {};
    \node (e110) [vertex] at (5,-1) {};
    
    \node (1e11) [vertex] at (3,2) {};
    \node (ee0e) [vertex] at (3,-2) {};
    
    \node (e0ee) [vertex] at (2,1) {};
    \node (11e1) [vertex] at (2,-1) {};
    
    \node (1101) [vertex] at (1,2) {};
    \node (1011) [vertex] at (1,-2) {};
    
    \node (e00e) [vertex] at (0,1) {};
    \node (e11e) [vertex] at (0,-1) {};
    
    \draw (7,1.4) node{$000\epsilon$};
    \draw (7,-1.4) node{$1000$};
    
    \draw (5.2,1.3) node{$0111$};
    \draw (5.2,-1.3) node{$\epsilon 110$};
    
    \draw (3,2.3) node{$1\epsilon 11$};
    \draw (3,-2.3) node{$\epsilon \epsilon0\epsilon $};
    
    \draw (2.2,1.3) node{$\epsilon 0 \epsilon\epsilon$};
    \draw (2.2,-1.3) node{$11\epsilon 1$};

    \draw (1,2.3) node{$1101$};
    \draw (1,-2.3) node{$1011$};
    
    \draw (-.2,1.3) node{$\epsilon 0 0\epsilon$};
    \draw (-.2,-1.3) node{$\epsilon 11\epsilon$};
    \end{tikzpicture}
    
    \caption{The $10$-vertex component, $H_1$, on the left satisfies $b_2(H_1)=4$, and $b_2(H_1 + K_2) = 4$. 
    Note that this is not possible for any graph $H$ with $b_2(H) < 4$.
    }
    \label{fig:plusk2}
\end{figure}

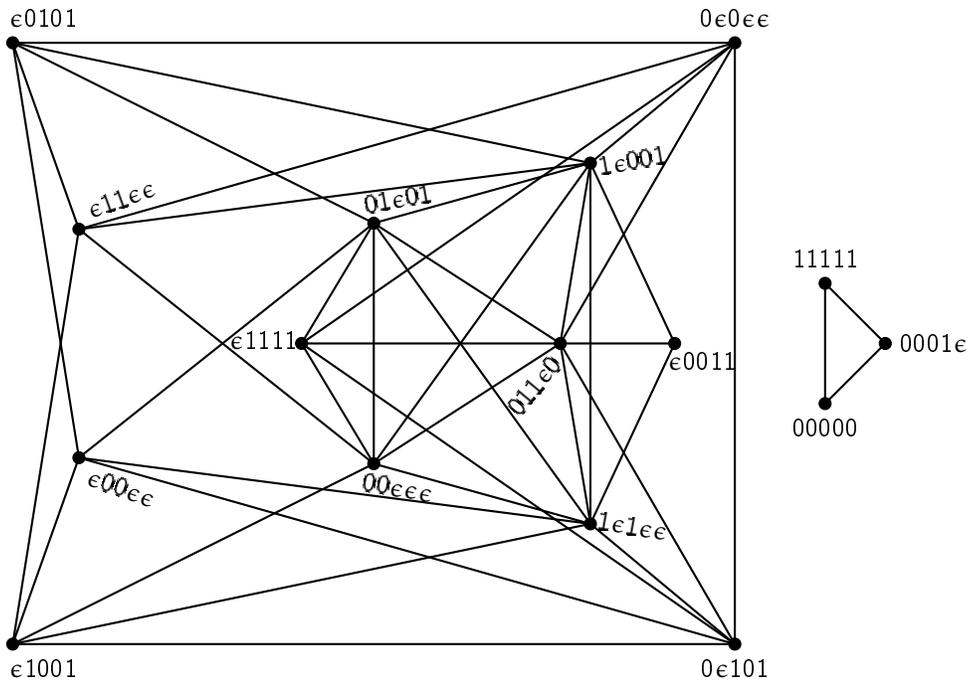
\begin{figure}
    \centering
    
\begin{tikzpicture}[scale=.8]

\draw [black!100,thick] (8.5,1) -- (8.5,-1);
\draw [black!100,thick] (9.5,0) -- (8.5,-1);
\draw [black!100,thick] (8.5,1) -- (9.5,0);

\draw[fill=black!100,draw=black!100] (8.5,1) circle (.1);
\draw[fill=black!100,draw=black!100] (8.5,-1) circle (.1);
\draw[fill=black!100,draw=black!100] (9.5,0) circle (.1);

\draw (10.3,0) node{$ 0001\epsilon$};
\draw (8.5,1.4) node{$ 11111$};
\draw (8.5,-1.4) node{$00000$};
\draw (7,5.4) node{$0\epsilon0\epsilon\epsilon$};
\draw (7,-5.4) node{$0\epsilon101$};
\draw (-4.5,-5.4) node{$\epsilon1001$};
\draw (-4.5,5.4) node{$\epsilon0101$};
\draw (-3.2,2.4) node[rotate=20]{$\epsilon11\epsilon\epsilon$};
\draw (-3.2,-2.4) node[rotate=-20]{$\epsilon00\epsilon\epsilon$};
\draw (-0.84,0.05) node{$\epsilon1111$};
\draw (1.4,2.4) node[rotate=10]{$01\epsilon01$};
\draw (1.4,-2.4) node[rotate=-10]{$00\epsilon\epsilon\epsilon$};
\draw (3.67,-.7) node[rotate=50]{$011\epsilon0$};
\draw (6.45,-0.3) node{$\epsilon0011$};
\draw (5.3,3.04) node[rotate=10]{$1\epsilon001$};
\draw (5.3,-3.04) node[rotate=-10]{$1\epsilon1\epsilon\epsilon$};

\draw[fill=black!100,draw=black!100] (4.6,3) circle (.1);
\draw[fill=black!100,draw=black!100] (4.6,-3) circle (.1);

\draw[fill=black!100,draw=black!100] (1,2) circle (.1);
\draw[fill=black!100,draw=black!100] (1,-2) circle (.1);

\draw[fill=black!100,draw=black!100] (-5,5) circle (.1);
\draw[fill=black!100,draw=black!100] (-5,-5) circle (.1);

\draw[fill=black!100,draw=black!100] (-3.9,1.9) circle (.1);
\draw[fill=black!100,draw=black!100] (-3.9,-1.9) circle (.1);

\draw[fill=black!100,draw=black!100] (7,5) circle (.1);
\draw[fill=black!100,draw=black!100] (7,-5) circle (.1);


\draw[fill=black!100,draw=black!100] (-0.2,0) circle (.1);
\draw[fill=black!100,draw=black!100] (4.1,0) circle (.1);
\draw[fill=black!100,draw=black!100] (6,0) circle (.1);

\draw [black!100,thick] (-0.2,0) -- (1,-2);
\draw [black!100,thick] (-0.2,0) -- (1,2);

\draw [black!100,thick] (4.1,0) -- (-0.2,0);
\draw [black!100,thick] (4.1,0) -- (6,0);

\draw [black!100,thick] (6,0) -- (4.6,3);
\draw [black!100,thick] (6,0) -- (4.6,-3);

\draw [black!100,thick] (4.1,0) -- (1,-2);
\draw [black!100,thick] (4.1,0) -- (1,2);

\draw [black!100,thick] (-5,5) -- (-3.9,1.9);
\draw [black!100,thick] (-5,5) -- (-3.9,-1.9);
\draw [black!100,thick] (-5,5) -- (1,2);
\draw [black!100,thick] (-5,5) -- (4.6,3);

\draw [black!100,thick] (-5,-5) -- (-3.9,1.9);
\draw [black!100,thick] (-5,-5) -- (-3.9,-1.9);
\draw [black!100,thick] (-5,-5) -- (1,-2);
\draw [black!100,thick] (-5,-5) -- (4.6,-3);

\draw [black!100,thick] (1,2) -- (1,-2);

\draw [black!100,thick] (4.6,3) -- (4.6,-3);
\draw [black!100,thick] (4.6,3) -- (1,2);
\draw [black!100,thick] (4.6,3) -- (1,-2);
\draw [black!100,thick] (4.6,3) -- (-3.9,1.9);
\draw [black!100,thick] (4.6,-3) -- (1,2);
\draw [black!100,thick] (4.6,-3) -- (-3.9,-1.9);
\draw [black!100,thick] (4.6,-3) -- (1,-2);

\draw [black!100,thick] (1,2) -- (-3.9,-1.9);
\draw [black!100,thick] (1,-2) -- (-3.9,1.9);

\draw [black!100,thick] (7,5) -- (7,-5);
\draw [black!100,thick] (7,5) -- (4.6,3);
\draw [black!100,thick] (7,5) -- (-3.9,1.9);
\draw [black!100,thick] (7,5) -- (-5,5);
\draw [black!100,thick] (7,5) -- (-0.2,0);
\draw [black!100,thick] (7,-5) -- (-0.2,0);
\draw [black!100,thick] (7,5) --(4.1,0);
\draw [black!100,thick] (7,-5) --(4.1,0);
\draw [black!100,thick] (7,-5) -- (4.6,-3);
\draw [black!100,thick] (7,-5) -- (-3.9,-1.9);
\draw [black!100,thick] (7,-5) -- (-5,-5);

\draw [black!100,thick] (4.1,0) -- (4.6,3);
\draw [black!100,thick] (4.1,0) -- (4.6,-3);

\end{tikzpicture}
    
    \caption{The $13$-vertex component, $H_2$, on the left satisfies $b_2(H_2)=5$, and $b_2(H_2 + K_3) = 5$. 
    Note that this is not possible for any graph $H$ with $b_2(H) < 5$.
    }
    \label{fig:plusk3}
\end{figure}

Next, we prove item~{\em (\ref{prop:item:K3})}.
We can embed $K_3$ into $B_5$ as the vertices \[\{00000,0001\eps,11111\}.\] We then choose $H_2$ to be the induced subgraph of $B_5$ on the vertex set of all vertices nonadjacent to $\{00000,0001\eps,11111\}$. By construction, $b_2(H_2+K_3)\le 5$, and indeed we can check that $b_2(H_2)\ge 5$. Since $b_2(H_2+K_3)\ge b_2(H_2)$, we have the desired result.

It is also possible to choose $H_2$ as a smaller induced subgraph of $B_5$. One option, shown in Figure \ref{fig:plusk3}, comes from using the vertex set
\[
\begin{array}{ccccccc}
     \{00\eps\eps\eps, & 0\eps0\eps\eps, & 0\eps101, & 01\eps01, & 011\eps0, & 1\eps1\eps\eps, & 1\eps001, \\
     \eps1111, & \eps00\eps\eps, & \eps11\eps\eps, & \eps0011, & \eps0101, & \eps1001\}.
\end{array}
\]

\end{proof}

\begin{ques}
For every nonempty graph $G$, is there some graph $H$ such that $b_2(G+H)=b_2(H)$?
\end{ques}

If the rank bound~\eqref{eq:ranklower} is tight for both $G$ and $H$, then it is easy to see that $b_2(G+H)=b_2(G)+b_2(H)$.
For the graph $H_1 + K_2$ in Figure~\ref{fig:plusk2}, we see that even when the rank bound is tight for one of $G$ or $H$,
it is possible to have $b_2(G+H)<b_2(G)+b_2(H)$.
It is also possible to have $b_2(G+H)=b_2(G)+b_2(H)$, such as when $G=K_3$ and $H=K_2$.
\begin{ques}
If the rank bound is tight for neither $G$ nor $H$, do we always have \[b_2(G+H)<b_2(G)+b_2(H)? \]
\end{ques}

Another disjoint union we can consider is $G + \bar{G}$, which can be thought of as partitioning the edge set of a clique into two pieces. Of course, $b_2(G + \bar{G}) \leq b_2(G) + b_2(\bar{G})$.
For small $n$, we have checked that $b_2(K_n) \leq b_2(G + \bar{G}) \leq n$. We can also look at the Nordhaus-Gaddum type inequality of bounding the quantity $b_2(G) + b_2(\bar{G})$. (See \cite{AOUCHICHE2013} for survey of similar type results.)

\begin{ques}
For a graph $G$ on $n$ vertices, do we always have $b_2(K_n)\le b_2(G+\bar{G})\le n$? 
\end{ques}

\begin{ques}
    What upper bounds can we put on $b_2(G) + b_2(\bar{G})$?
\end{ques}

\section{Cliques}\label{sec:clique}
In this section, we determine $b_2$ for odd cliques and disjoint unions of odd cliques.
We then turn to even cliques, proving that $K_{2n}$ has a perfect odd cover whenever $n \equiv 9 \pmod{12}$ and using a similar construction to show that $3K_{2n}$ has a perfect odd cover when $n \equiv 1 \pmod{4}$.
We also generalize a construction of Tomon for cliques of order $3^k - 1$ to find new perfect odd covers when $n = q^k - 1$ for certain prime powers $q$.
We conclude by proving certain necessary conditions for an even clique to have a perfect odd cover.

Two vertices $u$ and $v$ of $G$ are said are said to be {\em adjacent twins} if $N[u] = N[v]$. For the proofs of Theorems~\ref{thm:oddcliqueunion} and~\ref{theorem:even_partial}, we will make use of the following proposition from \cite[Proposition~6.1]{buchanan2022odd}:
\begin{prop}[\cite{buchanan2022odd}]\label{prop:oldmatching}
    If a graph $G$ contains a matching $M$ such that each edge $uv\in M$ is a pair of adjacent twins, then $r_2(G) = 2|M| + r_2(G-V(M))$.
\end{prop}

\subsection{Odd cliques}\label{sec:oddclique}
Here, we determine $b_2$ for odd cliques and disjoint unions of odd cliques.

\begin{thm}\label{thm:oddclique}
    $b_2(K_{2k+1})=k+1$.
\end{thm}

\begin{proof}
Let $u,~v_1,~\dots,~v_k,~w_1,~\dots,~w_k$ be the $2k+1$ vertices of $K_{2k+1}$. Let $G_1$ be the complete graph on $v_1,~\dots,~v_k$ and let $G_2$ be the complete graph on $w_1,~\dots,~w_k$. Then by Theorem~\ref{thm:disjoint_copies}, there exist $k$ complete bipartite graphs $B'_1,~\dots,~B'_k$
with parts $(X'_1,~Y'_1),~\dots,~(X'_k,~Y'_k)$ respectively such that: (i) $B'_1,~\dots,~B'_{k}$ form an odd cover of $G_1 + G_2$; and (ii) for all $1\le i\le k$, $X'_i=\{v_i,~w_i\}$. 

Now we construct an odd cover of $K_{2k+1}$ as follows. For all $1\le i\le k$, let
$$
\begin{aligned}
    &X_i=X'_i,\\
    &Y_i=Y'_i\cup \{u\};
\end{aligned}
$$
Further, let $X_{k+1}=\{v_1,~\dots,~v_k\}$ and $Y_{k+1}=\{w_1,~\dots,~w_k\}$. Then it is not hard to check that the $k+1$ complete bipartite graphs $B_1,~\dots,~B_{k+1}$ with parts $(X_1,~Y_1),~\dots,~(X_{k+1},~Y_{k+1})$ respectively form an odd cover of $K_{2k+1}$. This implies
$$
b_2(K_{2k+1})\le k+1.
$$
Therefore, by Theorem~\ref{thm:cliques from our first paper},
$$
b_2(K_{2k+1})=k+1.
$$
\end{proof}

\begin{thm}\label{thm:oddcliqueunion}
   Let $m_1,\dots, m_j$ be positive integers. Then $b_2(K_{2m_1+1}+\dots+K_{2m_j+1})=(m_1+\dots+m_j)+1$.
\end{thm}
\begin{proof}
    We first demonstrate the lower bound.
    Since $r_2(K_{2m_1+1}+\dots+K_{2m_j+1}) = \sum 2m_i$, it suffices to show that no perfect odd cover exists.
    Suppose, for the sake of contradiction, that there is a biclique $B = (X,Y)$ in a perfect odd cover.
    For each $K_{2m_i+1}$, there are $a_i$ vertices in $X$, $b_i$ vertices in $Y$, and $c_i$ vertices in $Z=V(K_{2m_1+1}+\dots+K_{2m_j+1})\setminus(X\cup Y)$.
    Every $K_{2m_i + 1}$ is an even core, and thus every $a_i$ and $b_i$ is even by Corollary~\ref{cor:evenint/evencore}.
    For each clique, we pair up vertices that are in the same one of $X$, $Y$, or $Z$ to get $a_i/2+b_i/2+\lfloor{c_i/2\rfloor}$ pairs of adjacent twins in $(K_{2m_1+1}+\dots+K_{2m_j+1})\symdif B$.
    Note that $m_i = a_i/2 + b_i/2 + \lfloor c_i / 2 \rfloor$, so we have $\sum m_i$ pairs of adjacent twins in $(K_{2m_1+1}+\dots+K_{2m_j+1})\symdif B$.
    By Proposition~\ref{prop:oldmatching}, the rank of this graph is $2 \sum m_i$, which implies $b_2((K_{2m_1+1}+\dots+K_{2m_j+1})\symdif B) \geq \sum m_i$, a contradiction.

    

    Now we will give a construction that provides a matching upper bound. Let $u_i,~v_{i,1},~\dots,~v_{i,m_i},~w_{i,1},~\dots,~w_{i,m_i}$ be the vertices of the $i^{th}$ complete graph $K_{2m_i+1}$, $1\le i\le j$. By Theorem~\ref{thm:disjoint_copies}, there exist $m_1+\dots+m_j$ complete bipartite graphs $B'_{i,k}$ with parts $(X'_{i,k},~Y'_{i,k})$, $1\le i\le j$, $1\le k\le m_i$, such that: {\it (i)} they form an odd cover of $2K_{m_1}+\dots+2K_{m_j}$ where one copy of $K_{m_i}$ is induced on the $v_{i,k}$'s and the other is induced on the $w_{i,k}$'s; and {\it (ii)} $X'_{i,k}=\{v_{i,k}, w_{i,k}\}$ for all $1\le i\le j$, $1\le k\le m_i$.
    
    Now we construct an odd cover of $K_{2m_1+1}+\dots+K_{2m_j+1}$ as follows. For all $1\le i\le j$, $1\le k\le m_i$, let $V_i=\{v_{i,1},~\dots,~v_{i,m_i}\}$, $W_i=\{w_{i,1},~\dots,~w_{i,m_i}\}$,
    $$
    \begin{aligned}
    &X_{i,k}=X_{i,k}',\\
    &Y_{i,k}=Y_{i,k}'\cup\{u_i\}\cup\left(\bigcup_{1\le t\le j,~t\not=i}V_t\right). 
    \end{aligned}
    $$
    Further, let
    $$
    \begin{aligned}
    &X_{0}=\cup_{1\le i\le t}V_i,\\
    &Y_{0}=\cup_{1\le i\le t}W_i.
    \end{aligned}
    $$
    Then it is not hard to check that the $m_1+\dots+m_j+1$ bicliques $(X_0,~Y_0),~(X_{i,k},~Y_{i,k})$, $1\le i\le j$, $1\le k\le m_i$, form an odd cover of $K_{2m_1+1}+\dots+K_{2m_j+1}.$ This shows that $b_2(K_{2m_1+1}+\dots+K_{2m_j+1})\le m_1+\dots+m_j+1$, which matches the lower bound.
\end{proof}


\subsection{Even cliques}\label{sec:evenclique}

To fully resolve Babai and Frankl's odd cover problem, it remains to determine which even values of $n$ satisfy $b_2(K_n)=n/2$. In other words, $K_n$ has a perfect odd cover for which even values of $n$? In \cite{buchanan2022odd}, it was shown that $b_2(K_{8m})=4m$ for all $m\ge 1$. Additionally, Radhakrishnan, Sen, and Vishwanathan \cite{radhakrishnan2000depth} provide an infinite family with $n\equiv 2\pmod{8}$ such that $b_2(K_n)=n/2$:

\begin{thm}[\cite{radhakrishnan2000depth}]\label{thm:2mod8}
    Let $q\equiv 3\pmod{4}$ be a prime power and let $n=2(q^2+q+1)$. Then, $b_2(K_n)=n/2$.
\end{thm}

Here, we demonstrate the existence of perfect odd covers for another family of positive density.

\begin{thm}\label{thm:18mod24}
   When $n=24k+18$, $b_2(K_n)=n/2$.
\end{thm}
\begin{proof}
It suffices to provide an upper bound construction. Following \cite{radhakrishnan2000depth}, we will use a \textit{pairs construction}, where the vertices are partitioned into pairs $(2i-1,2i)$ for $1\le i\le n/2$ and two vertices in the same pair appear in the exact same bicliques as well as the opposite partite set of each biclique in which they appear. To make a perfect odd cover, we construct an $(n/2)\times (n/2)$ matrix $M$ where row $i$ corresponds to the pair $(2i-1,2i)$ and column $j$ corresponds to biclique $(X_j,Y_j)$. In particular,
\[
m_{ij}=\begin{cases}
0 \hspace{3mm} \text{when } 2i-1 \text{ and } 2i \text{ do not appear in biclique } (X_j,Y_j),\\
1 \hspace{3mm} \text{when } 2i-1\in X_j, 2i\in Y_j, 
\\
-1 \hspace{3mm} \text {when } 2i\in X_j, 2i-1\in Y_j.
\end{cases}
\]

In order for this matrix $M$ to correspond to a perfect odd cover of $K_n$, the following conditions are necessary and sufficient. 
\begin{enumerate}[{\it (i)}]
    \item Each row must contain an odd number of $\pm 1$'s.
    \item For two distinct rows, there are an odd number of columns where one has $1$ and the other has $-1$.
    \item For two distinct rows, there are an odd number of columns where both have $1$ or both have $-1$.
\end{enumerate}
Since $3\mid (n/2)$, we construct $M$ as a block matrix 
\[
M=\begin{bmatrix}
    A&C&B\\
    C&B&A\\
    B&A&C
\end{bmatrix}
\]
where each block is an $(n/6) \times (n/6)$ matrix, $A$ consists entirely of $1$'s, $B$ consists entirely of $0$'s, and $C$ is defined as follows:
\[
c_{ij}=\begin{cases}
0 \hspace{3mm} \text{when } i=j,\\
1 \hspace{3mm}  \text{when } j>i \text{ and } j-i \text{ odd or } j<i \text{ and } j-i \text { even} , 
\\
-1 \hspace{3mm} \text{when } j>i \text{ and } j-i \text{ even or } j<i \text{ and } j-i \text { odd} .
\end{cases}
\]
For example, when $n=42$,

    \[C=\begin{bmatrix}
        0&1&-1&1&-1&1&-1\\
        -1&0&1&-1&1&-1&1\\
        1&-1&0&1&-1&1&-1\\
        -1&1&-1&0&1&-1&1\\
        1&-1&1&-1&0&1&-1\\
        -1&1&-1&1&-1&0&1\\
        1&-1&1&-1&1&-1&0
    \end{bmatrix}.\]

To verify that this construction yields a perfect odd cover, we note that each row of $M$ has $n/6$ $\pm 1$'s from $A$, and an additional $n/6-1$ from $C$, for a total of $n/3-1$, which is odd. Now for each pair of distinct rows, we must verify that there are the correct number of columns where they both have the same $\pm 1$ entry and the correct number of columns where they have different $\pm 1$ entries. 

First consider rows $i$ and $j$ which are both in the first $n/6$ rows, both in the next $n/6$ rows, or both in the last $n/6$ rows. Without loss of generality, we may assume that $i<j\le n/6$. Then the total number of columns where both have the same $\pm 1$ entry is $n/6$ (from the $A$ block), plus the number of columns in $C$ where rows $i$ and $j$ have the same $\pm 1$ entry. The number of columns where they have different $\pm 1$ entries is the number of columns in $C$ where rows $i$ and $j$ have different $\pm 1$ entries. As $n/6$ is odd, it suffices to verify that two distinct rows of $C$ have both an even number of columns where they are the same $\pm 1$ entry and an odd number of columns where they are different $\pm 1$ entries. Note that for two distinct rows of $C$, there are only two columns where one row or the other has a zero. As there are an odd number, $n/6$, of total columns, it means that having an even number of columns where the two rows have the same $\pm 1$ entry guarantees that there are also an odd number of columns where the two rows have different $\pm 1$ entries. Thus it suffices to check that there are an even number of columns such that rows $i$ and $j$ of $C$ have the same $\pm 1$ entry. If $j$ and $i$ have the same parity, this happens precisely for columns after column $j$ and those before column $i$, for a total of $n/6-(j-i+1)$. Since $n\equiv 18\pmod{24}$, this is even. If instead $j$ and $i$ have opposite parity, then this happens precisely for columns in between $i$ and $j$ for a total of $j-i-1$, which is again even.

Otherwise, we have without loss of generality, that $i$ is in the first $n/6$ rows and $j$ is in the next $n/6$ rows. For any column aside from the first $n/6$ columns, at least one of these rows has a $0$. For all the remaining columns, row $i$ corresponds to an $A$ block, so its only relevant entries are $1$'s. Thus it suffices to check that for row $j-n/6$ of $C$ that there are both an odd number of $1$'s and an odd number of $-1$'s. Indeed, any row of $C$ has $\frac{n/6-1}{2}$ of each. Since $n\equiv 18\pmod{24}$, this is odd, as desired.    
\end{proof}

Note that using the above construction where instead $n\equiv 6\pmod{24}$ gives a perfect odd cover of $3K_{n/3}$, meaning a perfect odd cover for $3K_{8k+2}$ whenever $k\in\mathbb{Z}_{\ge 0}$. To verify this, we note that the number of $\pm 1$'s in each row is still $n/3-1$, which is odd. Again for two rows $i$ and $j$ both in the first $n/6$ rows, both in the next $n/6$ rows, or both in the last $n/6$ rows, we have, by the exact same reasoning, that there are an odd number of columns where they have different $\pm 1$ values and an odd number of columns where they have the same $\pm 1$ value. This holds because $n/6$ is still odd. This means that the graph resulting from this pairs construction induces a $K_{n/3}$ on the vertices corresponding to the first $n/6$ rows, the vertices corresponding to the next $n/6$ rows, and the vertices corresponding to the last $n/6$ rows.
For two other rows, we may assume without loss of generality that $i$ is from the first $n/6$ rows and $j$ is from the next $n/6$. The only columns where both $i$ and $j$ have a $\pm 1$ entry are the first $n/6$. For all of these, row $i$ has a $1$. For row $j$, it suffices to check how many $1$'s and $-1$'s are in the corresponding row of $C$. There are $\frac{n/6-1}{2}$ of each. As $n\equiv 6\pmod{24}$, this is an even number. This corresponds to a lack of edges between the pair of vertices corresponding to row $i$ and the pair of vertices corresponding to row $j$, so indeed this pairs construction yields $3K_{n/3}$. In particular, this means $3K_{10}$ has a perfect odd cover even though $K_{10}$ does not.

As noted at the beginning of Section \ref{sec:disjoint}, it follows from Theorem \ref{thm:disjoint_copies} that $b_2(2K_{2n})=2n$ regardless of whether $b_2(K_{2n})=n$ or $b_2(K_{2n})=n+1$. Further, $b_2(2K_{2n_1}+\dots+2K_{2n_t})=2n_1+\dots+2n_t$, but little else is known about disjoint unions involving even cliques. Of particular interest is the following question. 



\begin{ques}
    For every value of $n$, is there some odd $t$ where $tK_{2n}$ has a perfect odd cover?
\end{ques}

Note that if such a $t$ exists, then any sufficiently large number of copies of $K_{2n}$ would have a perfect odd cover.

Returning to just one copy of a clique, Tomon \cite{tomon2023} showed that $b_2(K_n)=n/2$ whenever $n$ is of the form $3^k-1$. We include a proof here.

\begin{prop}[\cite{tomon2023}]\label{prop:tomon3^k}
    For any nonnegative integer $k$, $b_2(K_{3^k-1})=\frac{3^k-1}{2}$.
\end{prop}
\begin{proof}
    The rank bound gives the desired lower bound, so it suffices to provide a construction. We identify the vertex set of $K_n$ with $\mathbb{F}_3^k\setminus\{\vec{0}\}$. For each nonzero vector in $\mathbb{F}_3^k$, there is a pair of parallel hyperplanes not containing the origin which each have this vector as their normal vector. For each of these $n/2$ pairs, we form a biclique where the partite classes are the points on each hyperplane. We now check that this is indeed an odd cover of $K_n$:
    \begin{itemize}
        \item If $u,v\in \mathbb{F}_3^k\setminus\{\vec{0}\}$ are not scalar multiples of each other: The number of times edge $uv$ is covered is the number of vectors $a\in\mathbb{F}_3^k\setminus\{\vec{0}\}$ such that $\langle u,a\rangle =1, \langle v,a \rangle =2$. Without loss of generality, $u$ restricted to the first two coordinates is not a scalar multiple of $v$ restricted to the first two coordinates. To determine the possible values of $a$, we can arbitrarily assign the last $k-2$ coordinates and then solve a full-rank system of two linear equations to determine the first two coordinates. It is impossible for this to result in $a=0$, so there are a total of $3^{k-2}$ options, which is odd.
        \item If $u=2v$: For a vector $a\in\mathbb{F}_3^k\setminus\{\vec{0}\}$, we have that $\langle u,a \rangle \ne 0$ if and only if $\langle v,a \rangle \ne 0$. Furthermore, if $\langle u,a \rangle \ne 0$, then $\langle u,a \rangle \ne \langle v,a \rangle$, so $u$ and $v$ are in opposite partite classes. Thus the number of times edge $uv$ is covered is the number of vectors $a$ satisfying $\langle u,a \rangle =1\in\mathbb{F}_3$. Without loss of generality, the first coordinate of $u$ is nonzero. The last $k-1$ coordinates of $a$ can be chosen arbitrarily and then the first coordinate of $a$ is determined. It is impossible for this to result in $a=0$, so there are a total of $3^{k-1}$ options, which is odd.
    \end{itemize}
\end{proof}

We now generalize the idea of Tomon to find perfect odd covers for another infinite family of even cliques.


Whenever $b_2(K_m)=m/2$ and $q:=m+1$ is a prime power, we can give a construction for a perfect odd cover of $K_{q^k-1}$ as follows. We identify the vertex of set of $K_{q^k-1}$ with $\mathbb{F}_q^k\setminus\{\vec{0}\}$. For each nonzero vector in $\mathbb{F}_q^k$, there is a family of $q-1=m$ parallel hyperplanes not containing the origin which each have this vector as their normal vector. For each of these $\frac{q^k-1}{q-1}$ families, we form a set of $m/2$ bicliques where the partite classes are unions of hyperplanes. In particular, we take an odd cover of $K_m$ (with vertices labeled $1,2,\dots,m$) with $m/2$ bicliques and for each biclique $(X,Y)$, we create a corresponding biclique $(X',Y')$ in $K_{q^k-1}$ where a vertex is in $X'$ if and only if its inner product with the normal vector is in $X$ and in $Y'$ if and only if its inner product with the normal vector is in $Y$. This construction consists of $(\frac{q^k-1}{q-1})(m/2)=\frac{q^k-1}{2}$ bicliques and now we must check that it is actually an odd cover of $K_{q^k-1}$. 

    There are $\frac{q^k-1}{q-1}$ families of parallel hyperplanes. The only families for which we do not get an odd number of edges between two nonzero vectors $v,w$ are those where 
    \begin{itemize}
    \item $v,w$ are on the same hyperplane,
    \item $v$ is on a hyperplane through the origin,
    \item $w$ is on a hyperplane through the origin. 
    \end{itemize}
    We need to subtract the numbers of each of these types of families from $\frac{q^k-1}{q-1}$. By symmetry, the number of types of families where $v$ but not $w$ is on the hyperplane through the origin is the same as the other way around, so these terms do not affect the parity. Thus, we must subtract precisely the cases where $v,w$ lie on the same hyperplane (including the hyperplane through the origin). The points corresponding to $v,w$ lie on the same hyperplane with a given normal vector if and only if that normal vector is orthogonal to the nonzero vector $v-w$. There are $\frac{q^{k-1}-1}{q-1}$ such vectors up to scaling in $\mathbb{F}_q^k$, so the number of times $vw$ is covered has the same parity as
    \[
    \frac{q^k-1}{q-1}-\frac{q^{k-1}-1}{q-1}=q^{k-1},
    \]
    which is always odd.
We note that when $q>3$, we may form the bicliques so that this construction is not a pairs construction, distinguishing it from all of the previously presented perfect odd covers of even cliques.

However, this construction does not give us any new examples where $b_2(K_n)=n/2$. In particular, if $m\equiv 0\pmod{8}$ or $m\equiv 18\pmod{24}$, then $q^k-1$ is again either $0\pmod{8}$ or $18\pmod{24}$ for any $k$. If instead, we take $m$ to be one of the values where we know a perfect odd cover exists by Theorem \ref{thm:2mod8}, we must have some $t\equiv 3\pmod{4}$ such that $q=2t^2+2t+3$. If $t\equiv 1\pmod{3}$, then $q\equiv 19\pmod{24}$, so $q^k-1$ is either $0\pmod{8}$ or $18\pmod{24}$ for any $k$. Otherwise, $q$ is divisible by $3$, so either $q$ is not a prime power, or $q^k-1$ is a case already handled in Proposition \ref{prop:tomon3^k}.


Note that the families of perfect odd covers in Theorem \ref{thm:2mod8} and Proposition \ref{prop:tomon3^k}, due to \cite{radhakrishnan2000depth} and~\cite{tomon2023},
both contain infinitely many values of $n\equiv 2\pmod{24}$ such that $K_n$ has a perfect odd cover, the smallest of which are $2,26,242,266$. However, we have no examples of a perfect odd cover for $K_n$ when $n\equiv 10\pmod{24}$ and it was shown in \cite{buchanan2022odd} that $K_{10}$ has no perfect odd cover.

\begin{ques}
    Is it true that $b_2(K_n)=n/2$ whenever $n\equiv 2\pmod{24}$ but $b_2(K_n)=n/2+1$ whenever $n\equiv 10\pmod{24}$?
\end{ques}

We conclude this final section with a discussion of properties possessed by perfect odd covers of complete graphs.
Let us first note the implications of Corollary~\ref{cor:evenint/evencore} in the case of even complete graphs.
These graphs have full rank, and thus they contain no even cores.

\begin{cor}


    Suppose that $\mathcal{O}$ is a perfect odd cover of an even clique.
    If $I$ is a nonempty subset of vertices in the clique, then at least one partite set of a biclique in $\mathcal{O}$ has odd intersection with $I$.
\end{cor}

We now prove a number of more specific properties possessed by perfect odd covers of even cliques to support part \textit{(ii)} of Conjecture~\ref{conjecture:main}.
Figure~\ref{fig:even_partial} illustrates some of the information contained in the first two parts of the following theorem.

    \begin{thm}\label{theorem:even_partial}
        Let $\{B_i\}_{1\le i\le k}$ be a set of complete bipartite graphs forming a perfect odd cover of $K_{2k}$, where $B_i$ has parts $X_i$ and $Y_i$, $1\le i\le k$. The following conditions hold:
        \begin{enumerate}[{\it (i)}]
            \item\label{item:1,3mod4} If $k$ is odd, then $|X_i|,|Y_i|\equiv 1\pmod{4}$ for all $1\le i\le k$. If $k$ is even, then $|X_i|,|Y_i|\equiv 3\pmod{4}$ for all $1\le i\le k$.  
            \item\label{item:oddodd}  For all $i,j \in \{1, \ldots, k\}$ with $i\not=j$, $|X_i\cap X_j|$, $|X_i\cap Y_j|$ and $|Y_i\cap Y_j|$ are all odd.
            \item Let $U_i=X_i\cup Y_i$, $1\le i\le k$. Then each vertex is contained in odd number of $U_i$'s.
            \item  For any integer $s\equiv 2~\text{or}~3\pmod 4$, and any set $A\subseteq V(K_{2k})$ of size $s$, there exists an $i$ such that $|A\cap X_i|$ and $|A\cap Y_i|$ are odd.
        \end{enumerate}
    \end{thm}
    \begin{proof}
    We prove items~\textit{(i)}--\textit{(iv)} in order.
    \begin{enumerate}[{\it (i)}]
        \item  We begin by proving that $|X_i|,|Y_i|$ are odd. For a given $i$, let $Z_i:=V(K_{2k})\setminus(X_i\cup Y_i)$. Note that in the graph $K_{2k}\symdif B_i$, any two vertices in the same one of $X_i, Y_i$, or $Z_i$ are adjacent twins. Thus, if $|X_i|, |Y_i|$ are both even, we can form a perfect matching where each edge of the matching is between a pair of adjacent twins. Thus by Proposition \ref{prop:oldmatching}, $r_2(K_{2k}\symdif B_i)=2k$. Therefore, it takes at least a further $k$ bicliques to complete the odd cover of $K_{2k}$ and thus it will not be a perfect odd cover.
        Suppose instead that one of $|X_i|,|Y_i|$ is odd and the other is even. Without loss of generality, $|X_i|$ is odd and thus, so is $|Z_i|$. We can pair up all but one vertex in $X_i$, all vertices in $Y_i$, and all but one vertex in $Z_i$ to form a matching $M$ of $k-1$ edges where each edge is between a pair of adjacent twins. By Proposition \ref{prop:oldmatching}, $r_2(K_{2k}\symdif B_i)=2(k-1)+r_2(K_{2k}\symdif B_i-V(M))$. Note that $K_{2k}\symdif B_i-V(M)$ has just two vertices, but as one is in $X_i$ and the other is in $Z_i$, there is an edge between them. Thus $r_2(K_{2k}\symdif B_i-V(M))=2$, and $r_2(K_{2k}\symdif B_i)=2k$, meaning again it will take at least a further $k$ bicliques to complete the odd cover of $K_{2k}$. The only remaining option for $K_{2k}$ to have a perfect odd cover is if $|X_i|,|Y_i|$ odd for each biclique in the odd cover.

        Now we determine $|X_i|,|Y_i|\pmod{4}$. Consider the graph $H:=K_{2k}\symdif(X_i,Y_i)$.  Both the induced graphs $H[X_i\cup Z_i]$ and $H[Y_i\cup Z_i]$ are complete graphs with an odd number of vertices. Thus, any vertex of $H$ has an even number of neighbors in each of $X_i\cup Z_i, Y_i\cup Z_i$, so $X_i\cup Z_i, Y_i\cup Z_i$ are both even cores of $H$. If either $|X_i\cup Z_i|$ or $|Y_i\cup Z_i|$ is $3\pmod{4}$, then the complete graph on that vertex set has an odd number of edges, so by Corollary \ref{cor:evenint/evencore}, we would then get that an odd cover of $H$ requires at least $\frac{r_2(H)}{2}+1$ bicliques. Note that $r_2(H)\ge r_2(K_{2k})-2$, so an odd cover of $H$ would require at least $r_2(K_{2k})/2$ bicliques, meaning an odd cover of $K_{2k}$ would require more than that.
    Therefore, both $|X_i\cup Z_i|,|Y_i\cup Z_i|$, which are odd, must be $1\pmod{4}$. If $k$ odd, this yields $|X_i|,|Y_i|\equiv 1\pmod{4}$, while if $k$ even, this yields $|X_i|,|Y_i|\equiv 3\pmod{4}$.
        \item  For a given pair $i\ne j$, vertices that are in both the same one of $X_i,Y_i,Z_i$ and the same one of $X_j,Y_j,Z_j$ are adjacent twins. Suppose that not all of $|X_i\cap X_j|, |X_i\cap Y_j|, |Y_i\cap X_j|$, and $|Y_i\cap Y_j|$ are odd. Without loss of generality, there are five ways for this to happen: 
        \begin{itemize}
            \item[(I)] All are even.
            \item[(II)] All but $|Y_i\cap Y_j|$ are even.
            \item[(III)] $|X_i\cap X_j|, |X_i\cap Y_j|$ are even and the other two are odd.
            \item [(IV)]$|X_i\cap X_j|, |Y_i\cap Y_j|$ are odd and the other two are even.
            \item[(V)] All but $|Y_i\cap Y_j|$ are odd.
        \end{itemize}
In each case, we will form a matching of adjacent twins and then apply Proposition \ref{prop:oldmatching} to determine $r_2(K_{2k}\symdif B_i\symdif B_j)$. In Case (I), we obtain a matching with $k-2$ edges. The remaining vertices are in $|Z_i\cap X_j|, |Z_i\cap Y_j|, |X_i\cap Z_j|, |Y_i\cap Z_j|$, so they constitute a $C_4$, which has rank $2$. Thus, we get $r_2(K_{2k}\symdif B_i\symdif B_j)=2(k-2)+2=2k-2$, meaning at least $k-1$ further bicliques are required to complete the odd cover of $K_{2k}$.
In Case (II), we obtain a matching with $k-2$ edges. The remaining vertices, which are in $|Z_i\cap X_j|, |Y_i\cap Y_j|, |X_i\cap Z_j|, |Z_i\cap Z_j|$ constitute a $C_3$ with a pendant edge, which has rank $4$. Thus, we get $r_2(K_{2k}\symdif B_i\symdif B_j)=2(k-2)+4=2k$, meaning at least $k$ further bicliques are required to complete the odd cover of $K_{2k}$.
In Case (III), we obtain a matching with $k-2$ edges. The remaining vertices, which are in $|Y_i\cap X_j|, |Y_i\cap Y_j|, |Y_i\cap Z_j|,$ and $|X_i\cap Z_j|$ constitute a $K_{1,2}$ with an isolated vertex, which has rank $2$. Thus, we get $r_2(K_{2k}\symdif B_i\symdif B_j)=2(k-2)+2=2k-2$, meaning at least $k-1$ further bicliques are required to complete the odd cover of $K_{2k}$.
In Case (IV), we obtain a matching with $k-1$ edges. The remaining vertices are in $|X_i\cap X_j|, |Y_i\cap Y_j|$, so they constitute a $K_2$, which has rank $2$, Thus, $r_2(K_{2k}\symdif B_i\symdif B_j)=2(k-1)+2=2k$, meaning at least $k$ further bicliques are required to complete the odd cover of $K_{2k}$.
In Case (V), we obtain a matching with $k-3$ edges. The remaining vertices are in $|X_i\cap X_j|, |X_i\cap Y_j|, |X_i\cap Z_j|, |Y_i\cap X_j|, |Z_i\cap X_j|, |Z_i\cap Z_j|$, which form a graph of rank $4$, in particular, the complement of the graph consisting of a path of length $4$ and an isolated vertex;
see Figure \ref{fig:caseIIv}.
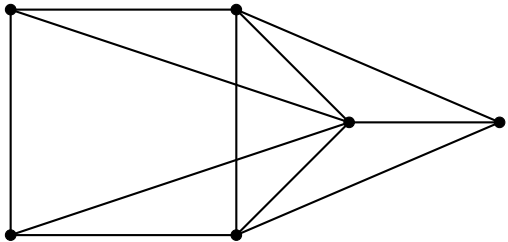
\begin{figure}
    \centering
    \begin{tikzpicture}
    [vertex/.style={circle, draw=black!100, fill=black!100, inner sep=0pt, minimum size=4pt},
    every edge/.style={draw, black!100, thick}]
    
    \draw [black!100,thick] (6,1.5) -- (3,1.5)--(3,-1.5)--(6,-1.5)--(6,1.5);

    \draw [black!100,thick] (6,1.5)--(7.5,0)--(6,-1.5);
    \draw [black!100,thick] (3,1.5) -- (7.5,0)--(3,-1.5);
    \draw [black!100,thick] (7.5,0)--(9.5,0);
    \draw [black!100,thick] (6,1.5)--(9.5,0)--(6,-1.5);

    \node (1) [vertex] at (7.5,0) {};
    \node (2) [vertex] at (9.5,0) {};
    
    \node (3) [vertex] at (6,1.5) {};
    \node (4) [vertex] at (6,-1.5) {};
    
    \node (5) [vertex] at (3,1.5) {};
    \node (6) [vertex] at (3,-1.5) {};
    
    \end{tikzpicture}    
    
    \caption{The graph obtained in Case (V) of the proof of Theorem \ref{theorem:even_partial} \textit{(ii)} after removing a maximum matching of adjacent twins from $K_{2k}\symdif B_i\symdif B_j$.}
    \label{fig:caseIIv}
\end{figure}Thus, $r_2(K_{2k}\symdif B_i\symdif B_j)=2(k-3)+4=2k-2$, meaning at least $k-1$ further bicliques are required to complete the odd cover of $K_{2k}$.
Thus, it is only possible to form a perfect odd cover of $K_{2k}$ if $|X_i\cap X_j|, |X_i\cap Y_j|, |Y_i\cap X_j|$, and $|Y_i\cap Y_j|$ are all odd.

        \item  Each vertex has odd degree in $K_{2k}$, so must be contained in an odd number of edges across all bicliques. However, each vertex has odd degree in each biclique in which it appears, so it must appear in an odd number of bicliques.
        \item  Note that the number of edges in $A$ is $\binom{s}{2}$, which is an odd number. To cover each edge in $A$ odd number of times, we need odd number of edges. Hence there exists a biclique $B_i$ with odd number of edges in $A$. This is only possible when $|X_i\cap A|$ and $|Y_i\cap A|$ are both odd.
    \end{enumerate}
\end{proof}

\begin{figure}
    \centering
    \begin{tikzpicture}
        \draw (0,0) ellipse (1.3 and .5);
        \node (xi) at (-1.8,0) {$X_i$};
        \draw (1,-1) ellipse (.5 and 1.3);
        \node (yj) at (1,.8) {$Y_j$};
        \draw (0,-2) ellipse (1.3 and .5);
        \node (yi) at (-1.8, -2) {$Y_i$};
        \draw (-1,-1) ellipse (.5 and 1.3);
        \node (xj) at (-1, .8) {$X_j$};

        \node (0) at (0, 0) {\LARGE$\star$};
        \node (1) at (1,0) {\LARGE$\star$};
        \node (2) at (1,-1) {\LARGE$\star$};
        \node (3) at (1,-2) {\LARGE$\star$};
        \node (4) at (0,-2) {\LARGE$\star$};
        \node (5) at (-1,-2) {\LARGE$\star$};
        \node (6) at (-1,-1) {\LARGE$\star$};
        \node (7) at (-1,0) {\LARGE$\star$};
    \end{tikzpicture}
    \caption{A pair of bicliques $(X_i, Y_i), (X_j, Y_j)$ in a perfect odd cover of an even clique (see items~\textit{(i)} and~\textit{(ii)} of Theorem~\ref{theorem:even_partial}). A $\star$ denotes an odd number of vertices. }
    \label{fig:even_partial}
\end{figure}
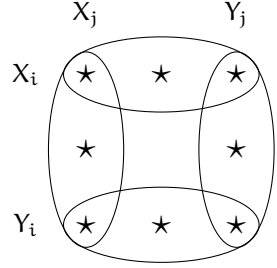

By Theorem \ref{theorem:even_partial} \textit{(ii)}, we note that both $|X_i\cap X_j|$ and $|X_i\cap Y_j|$ must be positive, so no biclique in a perfect odd cover of $K_{2k}$ is a star.

With the exception of some of the finite field constructions, all constructions of perfect odd covers for even cliques so far can be expressed as pairs constructions. Let $K_{2n}$ be a clique on vertices $v_1,\dots,v_{2n}$. Recall that a pairs construction of $K_{2n}$ is a perfect odd cover of $K_{2n}$ consisting of $n$ bicliques $(X_1,Y_1),\dots,(X_n,Y_n)$ such that for all $1\le i\le n$, $1\le j\le n$ we have $v_{2i-1}\in X_j$ if and only if $v_{2i}\in Y_j$ and, similarly,  $v_{2i-1}\in Y_j$ if and only if $v_{2i}\in X_j$. To solve part \textit{(ii)} of Conjecture~\ref{conjecture:main}, it suffices to show that there is no perfect odd cover of $K_{8k+4}$ and $K_{8k+6}$. It is easy to show that $K_{8k+4}$ and $K_{8k+6}$ have no pairs construction. In fact we can prove the following slightly stronger result. Given a perfect odd cover of $K_{2n}$ consisting of bicliques $(X_1,Y_1),\dots,(X_n,Y_n)$, we say that two vertices $u,v$ in $K_{2n}$ are \emph{of the same type} if for every $1\le i\le n$ we have $u\in X_i\cup Y_i$ if and only if $v\in X_i\cup Y_i$. Clearly, this defines an equivalence relation. With a pairs construction, vertices of $K_{2n}$ can be partitioned into $n$ pairs of vertices where each pair of two vertices are of the same type. 

\begin{thm}
For every integer $k\ge 0$, let $2n=8k+4$ (or $8k+6$) and let $v_1,\dots, v_{2n}$ be vertices of $K_{2n}$. Then $K_{2n}$ has no perfect odd cover such that for every $1\le i\le n$, $v_{2i-1}$ and $v_{2i}$ are of the same type.
\end{thm}

\begin{proof}
We first present the proof for $K_{8k+4}$. Let $v_1,\dots,v_{8k+4}$ be the vertices of $K_{8k+4}$ and assume (for a contradiction) that there is a perfect odd cover of $K_{8k+4}$ consisting of bicliques $(X_1,Y_1),\dots,$ $(X_{4k+2}, Y_{4k+2})$ such that for every $1\le i\le 4k+2$, $v_{2i-1}$ and $v_{2i}$ are of the same type. Let $A=\{v_{2i}~\big|~1\le i\le 4k+2\}$. By item (i) of Theorem~\ref{theorem:even_partial}, we know that $|X_i\cup Y_i|\equiv 2 \pmod 4$ for every $1\le i\le 4k+2$. By definition we have $|A\cap X_i|+|A\cap Y_i|=|A\cap(X_i\cup Y_i)|=|X_i\cup Y_i|/2\equiv 1\pmod 2$ for every $1\le i\le 4k+2$. However, by item (iv) of Theorem~\ref{theorem:even_partial}, there exists $1\le i\le 4k+2$ such that $|A\cap X_i|+|A\cap Y_i|\equiv 0\pmod 2$, a contradiction. This completes the proof for $K_{8k+4}$. The proof for $K_{8k+6}$ is similar.
\end{proof}

\begin{ques}
    Are there any other values of $n$ where $K_n$ has a perfect odd cover besides the following:
    \begin{enumerate}[{\it (i)}]
        \item[(i)] $n=3^k-1$ for a positive integer $k$.
        \item[(ii)] $n$ divisible by $8$.
        \item[(iii)] $n\equiv 18\pmod{24}$.
        \item[(iv)] $n=2(q^2+q+1)$ for a prime power $q\equiv 3\pmod{4}$.
    \end{enumerate}

    In particular, are there any $n\equiv 4\pmod{8}$ or $n\equiv 6\pmod{8}$ such that $K_n$ has a perfect odd cover?
\end{ques}

\section{Acknowledgments}
A.C.~was supported by the Institute for Basic Science (IBS-R029-C1). M.Y.~was supported by the University of Denver's Professional Research Opportunities for Faculty Fund 80369-145601. The authors would like to thank Kevin Hendrey and Jason O'Neill for helpful discussions.


\begin{thebibliography}{1}

\bibitem{AOUCHICHE2013}
Mustapha Aouchiche and Pierre Hansen.
\newblock {\em A survey of Nordhaus–Gaddum type relations}.
\newblock {\em Discrete Applied Mathematics},
161(4):466--546, 2013


\bibitem{bf}
L\'{a}szl\'{o} Babai and P\'{e}ter Frankl.
\newblock {\em Linear Algebra Methods in Combinatorics: Part 1}.
\newblock Department of Computer Science, The University of Chicago, 1988 (preliminary version).


\bibitem{beaudrap}
Niel de~Beaudrap~(https://mathoverflow.net/users/3723/niel-de-beaudrap).
\newblock Decomposition of graphs as symmetric differences of copies of
  ${K}_{a,b}$.
\newblock MathOverflow.
\newblock URL:https://mathoverflow.net/q/76043 (version: 2011-10-27).

\bibitem{buchanan2022odd}
Calum Buchanan, Alexander Clifton, Eric Culver, Jiaxi Nie, Jason O'Neill, Puck
  Rombach, and Mei Yin.
\newblock Odd covers of graphs.
\newblock {\em Journal of Graph Theory}, 104(2):420--439, 2023.

\bibitem{buchanan2022subgraph}
Calum Buchanan, Christopher Purcell, and Puck Rombach.
\newblock Subgraph complementation and minimum rank.
\newblock {\em The Electronic Journal of Combinatorics}, 29(1):P1--38, 2022.


\bibitem{godsil2001chromatic}
Chris~D Godsil and Gordon~F Royle.
\newblock Chromatic number and the $2$-rank of a graph.
\newblock {\em Journal of Combinatorial Theory, Series B}, 81(1):142--149,
  2001.

\bibitem{GP}
Ron Graham and Henry Pollak.
\newblock On embedding graphs in squashed cubes.
\newblock {\em Graph Theory and Appl., Springer Lecture Notes in Math.},
  303:99--110, 1972.

\bibitem{Hal35}
Philip Hall.
\newblock On representatives of subsets.
\newblock {\em Journal of the London Mathematical Society}, 10(1):26--30, 1935.


\bibitem{LT24} Imre Leader and Ta Sheng Tan.
\newblock Odd covers of complete graphs and hypergraphs
\newblock arxiv:2408.05053 (2024)

\bibitem{radhakrishnan2000depth}
Jaikumar Radhakrishnan, Pranab Sen, and Sundar Vishwanathan.
\newblock Depth-3 arithmetic circuits for ${S}_{n}^{2}(x)$ and extensions of
  the {G}raham-{P}ollack theorem.
\newblock In {\em International Conference on Foundations of Software
  Technology and Theoretical Computer Science}, pages 176--187. Springer, 2000.

\bibitem{tomon2023}
Istv\'an Tomon.
\newblock private communication.

\end{thebibliography}
\end{document}